\documentclass[11pt,reqno]{amsart}

\usepackage{amsmath, amsfonts, amsthm, amssymb}

\newtheorem{Theorem}{Theorem}[section]
\newtheorem{Lemma}{Lemma}[section]
\newtheorem{Proposition}{Proposition}[section]
\newtheorem{Corollary}{Corollary}[section]

\theoremstyle{definition}

\theoremstyle{remark}
\newtheorem{Remark}{Remark}[section]

\numberwithin{equation}{section}

\renewcommand{\r}{\rho}

\renewcommand{\u}{{\bf u}}

\newcommand{\R}{{\mathbb R}}
\newcommand{\Dv}{{\rm div}}
\newcommand{\tr}{{\rm tr}}

\newcommand{\dl}{\delta}
\newcommand{\E}{{\mathcal E}}

\def\f{\frac}
\renewcommand{\O}{\Omega}

\def\D{\Delta }

\def\hf1{^\f{1}{1-\xi^2}}

\def\be{\begin{equation}}
\def\en{\end{equation}}
\def\bs{\begin{split}}
\def\es{\end{split}}

\newcommand{\F}{{\mathtt F}}

\begin{document}

\author{Xianpeng Hu and Dehua Wang}
\address{Courant Institute of Mathematical Sciences, New York University,  New York, NY 10012.}
\email{xianpeng@cims.nyu.edu}

\address{Department of Mathematics, University of Pittsburgh, Pittsburgh, PA 15260.}
\email{dwang@math.pitt.edu}

\title[Compressible Viscoelastic Fluids]
{The initial-boundary value problem for the compressible viscoelastic fluids}

\date{December 14, 2010}%\today}

\begin{abstract}
The global existence of strong solution to the initial-boundary value problem of the
three-dimensional compressible viscoelastic fluids near
equilibrium is established in a bounded domain. Uniform estimates in $W^{1,q}$ with $q>3$ on the
density and  deformation gradient   are also obtained. All the
results apply to the two-dimensional case.
\end{abstract}

\maketitle

\section{Introduction}
Elastic solids and viscous fluids are two extremes of material
behavior. Viscoelastic fluids  show intermediate behavior with
some remarkable phenomena  due to their ``elastic" nature. These
fluids exhibit a combination of both fluid and solid
characteristics, keep memory of their past deformations, and their
behaviour is a function of these old deformations. Viscoelastic
fluids have a wide range of applications  and hence have received
a great deal of interest. Examples and applications of
viscoelastic fluids include from oil, liquid polymers, mucus,
liquid soap,  toothpaste,  clay,   ceramics,  gels,  some types of
suspensions, to  bioactive fluids,  coatings and drug delivery
systems for controlled drug release, scaffolds for tissue
engineering, and viscoelastic blood flow  past valves; see
\cite{ KM} for more applications. For the viscoelastic
materials, the competition between the kinetic energy and the
internal elastic energy through the special transport properties
of their respective internal elastic variables makes the materials
more untractable in  understanding their behavior, since any
distortion of microstructures, patterns or configurations in the
dynamical flow will involve the deformation tensor. For classical
simple fluids,  the internal energy can be determined solely by
the determinant of the deformation tensor; however, the internal
energy of complex fluids carries all the information of the
deformation tensor. The interaction between the microscopic
elastic properties and the macroscopic fluid motions  leads to the
rich and complicated rheological phenomena in viscoelastic fluids,
and also causes  formidable analytic and numerical challenges in
mathematical analysis. The equations of the compressible viscoelastic fluids of Oldroyd type
(\cite{Oldroyd1, Oldroyd2}) in three spatial dimensions take the
following form \cite{Joseph, LW, RHN}:
\begin{subequations} \label{e1e}
\begin{align}
&\varrho_t +\Dv(\varrho\u)=0,\label{e1e1}\\
&(\varrho\u)_t+\Dv\left(\varrho\u\otimes\u\right)-\mu\D \u-(\mu+\lambda)\nabla\Dv\u+\nabla P(\varrho)
=\Dv(\varrho \F\F^\top),\label{e1e2}\\
&\F_t+\u\cdot\nabla\F=\nabla\u \, \F,\label{e1e3}
\end{align}
\end{subequations}
where $\varrho$ stands for the density, $\u\in \R^3$ the velocity, and
$\F\in M^{3\times 3}$ (the set of  $3\times 3$ matrices)  the
deformation gradient. The viscosity coefficients $\mu>0$ and $\lambda$ satisfy
$$\mu>0,\quad\textrm{and}\quad 2\mu+3\lambda>0,$$ which ensures that
the operator $-\mu\D \u-(\mu+\lambda)\nabla\Dv\u$ is a strongly elliptic operator. The term $P(\varrho)$ represents the pressure for the barotropic case and is an increasing and convex smooth function of $\varrho >0$ with $P(1)>0$. The symbol $\otimes$
denotes the Kronecker tensor product and $\F^\top$ means the
transpose matrix of $\F$. As usual we call equation \eqref{e1e1}
the continuity equation. For system \eqref{e1e}, the corresponding
elastic energy is chosen to be  the special form of the Hookean
linear elasticity:
$$W(\F)=\frac12|\F|^2,$$
which, however, does not reduce the essential difficulties for
analysis. The methods and results of this paper can be applied to
more general cases.

In this paper, we consider equations \eqref{e1e} in the three-dimensional bounded domain $\O\subset\R^3$ with sufficiently smooth boundary,  subject to the following initial and boundary conditions:
\begin{equation}\label{IC1}
\begin{cases}
(\varrho, \u, \F)|_{t=0}=(\varrho_0(x), \u_0(x), \F_0(x)), \quad x\in\O,\\
\u|_{\partial\O}=0,
\end{cases}
\end{equation}
and we are  interested  in the existence and uniqueness of strong
solution to the initial-boundary value problem \eqref{e1e}-\eqref{IC1} near
its equilibrium state.  Here
the {equilibrium state} of the system \eqref{e1e} is defined as:
$\varrho$ is a positive constant (for simplicity, $\varrho=1$), $\u=0$, and
$\F=I$ (the identity matrix in $M^{3\times 3}$).  We denote the perturbations of the density,  the velocity, and the deformation gradient about the equilibrium by $\r$,  $\u$, and $E$, respectively, that is,
 $$\varrho=1+\r, \qquad \F=I+E.$$
 Then, system \eqref{e1e} becomes equivalently
\begin{subequations} \label{e1}
\begin{align}
&\r_t +\Dv(\r\u)+\Dv\u=0,\label{e11}\\
&((1+\r)\u)_t+\Dv\left((1+\r)\u\otimes\u\right)-\mu\D \u-(\mu+\lambda)\nabla\Dv\u+\nabla P(1+\r) \notag \\
&\qquad\qquad\qquad\qquad =\Dv((1+\r)(I+E)(I+E)^\top),\label{e12}\\
&E_t+\u\cdot\nabla E=\nabla\u \,E+\nabla\u,\label{e13}
\end{align}
\end{subequations}
with the initial and boundary conditions:
\begin{equation}\label{IC}
\begin{cases}
(\r, \u, E)|_{t=0}=(\r_0(x), \u_0(x), E_0(x))=(\varrho_0(x)-1, \u_0(x), \F_0(x)-I),  \quad x\in\O,\\
\u|_{\partial\O}=0.
\end{cases}
\end{equation}
By a {\em strong solution}, we mean a triplet $(\r, \u, E)$
satisfying \eqref{e1} almost everywhere with the initial and boundary conditions
\eqref{IC}, in particular, $(\r,\u,E)(\cdot, t)\in W^{1,q}\times \left(W^{2,q}\right)^3\times\left(W^{1,q}\right)^{3\times 3}$,
$q\in(3,\infty)$ %in some time interval $[0,T]$ for $T>0$
 in this paper.

When the density  is a constant, system \eqref{e1e} governs the
homogeneous incompressible viscoelastic fluids, and there exist
rich results in the literature for the global existence of
classical solutions (namely in $H^3$ or other functional spaces
with much higher regularity); see \cite{CM, CZ, LLZH2, LLZH,
LZ2010, LZ, LLZ, LZP, LW} and the references therein. When the
density  is not a constant, the question related to existence
becomes much more complicated. In \cite{HW1} the authors
considered the global existence of classical solutions of small
perturbation near its equilibrium for the compressible
viscoelastic fluids in critical spaces (a functional space in
which the system is scaling invariant), see also \cite{QZ}. For
the local existence of strong solutions with large initial data,
see \cite{HW}. One of the main difficulties in proving the global
existence for compressible viscoelastic fluids is the lacking of
the dissipative estimate for the deformation gradient. To overcome
this difficulty, the authors in \cite{LLZH} introduced an
auxiliary function to obtain the dissipative estimate, while the
authors in \cite{LZ} directly deal with the quantities such as
$\Delta \u+\Dv \F$. Those methods can provide them with some good
estimates, partly because of their high regularity of $(\u,\F)$.
However, in this paper, we deal with the strong solution with much
less regularity for $(\u,\F)\in W^{2,q}\times W^{1,q}$,
$q\in(3,\infty)$, hence those methods do not apply (more
precisely, we can not apply the standard energy method in our
situation). For our purpose, we first obtain a uniform estimate
for the linearized momentum equation using the maximal
regularities of the Stokes equations and parabolic equations, and
then we find that a combination between the velocity and the
convolution of the divergence of the deformation gradient with the
fundamental solution of Laplace operator will develop some good
dissipative estimates which are very useful for the global
existence.

When the initial-boundary value problem  \eqref{e1}-\eqref{IC} is considered, several difficulties need to be addressed:
\begin{enumerate}
 \item For the initial-boundary value problem, the boundary condition can not be prescribed as zero for the density or the deformation gradient, and hence taking the derivatives of equations up to arbitrary orders and integrating by parts can not be applied. To overcome this difficulty, we need to establish a uniform estimate for the linearized momentum equation with the help of the standard $L^p-L^q$ estimates for the Stokes equations and the variants for the heat euqation;
\item The functional framework now is the general $L^q$ space with $q>3$. This setting obviously exclude the classical methods which are the key tools in \cite{HW, LLZH, LLZ, LZP, QZ}. To handle this difficulty, the classical Poincare's inequality is used and a new conserved quantity for \eqref{e1e} will be needed.
\end{enumerate}

One key observation of this work is that under the condition
\eqref{curl1} (see Section 2) on the curl of the deformation gradient initially,
not only the curl of the deformation gradient at any positive time
is a higher order term, but also is the sum of the gradient of
density and the divergence of the deformation gradient (see
\eqref{m}), which motivates us to use the $L^p-L^q$ estimate of
Stokes equations. Furthermore, the divergence of the difference
between the deformation gradient and its transpose is also a
higher order term, see \eqref{q1}.

The viscoelastic fluid system \eqref{e1e} can be regarded as a
combination of the inhomogeneous  compressible Navier-Stokes
equations with the source term $\Dv(\r\F\F^\top)$ and the equation
\eqref{e1e3}.  For the global existence of classical solutions
with small perturbation near an equilibrium for the compressible
Navier-Stokes equations, we refer the reader to \cite{MT1, MT, AI}
and the references cited therein. We remark that,  for the
nonlinear inviscid elastic systems, the existence of solutions
was established by Sideris-Thomases in \cite{ST} under the null
condition; see also \cite{ST2} for a related discussion.

The existence of global weak solutions with large initial data of
\eqref{e1e} is still an outstanding open question. In this
direction for the homogeneous incompressible viscoelastic fluids,
when the contribution of the strain rate (symmetric part of
$\nabla\u$) in the constitutive equation is neglected,
Lions-Masmoudi in \cite{LM} proved the global existence of weak
solutions with large initial data for the Oldroyd model. Also
Lin-Liu-Zhang in \cite{LLZ} proved the existence of global weak
solutions with large initial data for the incompressible
viscoelastic fluids when the velocity satisfies the Lipschitz
condition. When dealing with the global existence of weak
solutions of the viscoelastic fluid system \eqref{e1e} with large
data, the rapid oscillation of the density and the
non-compatibility between the quadratic form and the weak
convergence are  two of the major difficulties.

The rest of the paper is organized as follows. In Section 2, we
recall several intrinsic properties of the system \eqref{e1e} and
also show a new conserved quantity. In Section 3, we introduce the
functional spaces and state our main results, including the local
existence and uniqueness of the strong solution to the system
\eqref{e1}-\eqref{IC}, as well as the global existence. In Section
4, we prove the uniform estimate for the linearized momentum
equation. In Section 5, the main goal is to obtain a series of
uniform estimates for the gradient of the density and the gradient
of the deformation gradient. In Section 6, we establish some
uniform in time {\it a priori} estimates on the dissipation of the
deformation gradient and the density, and finally finish the proof
of our main theorem.

\bigskip

\section{Some Intrinsic Properties of System \eqref{e1e}}

In this section, we  recall some intrinsic properties of system
\eqref{e1e} and introduce a new conserved quantity. First, for the
system \eqref{e1e} it was proved that (see Lemmas 6.1 and 6.2 in
Hu-Wang \cite{HW1}, and  also \cite{QZ}).
\begin{Proposition}\label{p1}
Assume that $(\varrho, \u, \F)$ is a solution of 	 system
\eqref{e1e}. Then the following identities
\begin{equation}\label{Det1}
 \varrho \det(\F)=1,
\end{equation}
\begin{equation}\label{Dv1}
\Dv(\varrho\F^\top)=0,
\end{equation}
and
\begin{equation}\label{curl}
\F_{lk}\nabla_l \F_{ij}=\F_{lj}\nabla_l \F_{ik}
\end{equation}
hold for all time $t> 0$ if they are satisfied initially.
\end{Proposition}

Note that \eqref{curl} can be interpreted in terms of  the perturbation
about the equilibrium as
\begin{equation}\label{curl1}
\partial_{x_k}E_{ij}-\partial_{x_j}E_{ik}=E_{lj}\nabla_l E_{ik}-E_{lk}\nabla_l
E_{ij}.
\end{equation}
Similarly, \eqref{Det1} can be also interpreted in terms of  the perturbation
about the equilibrium as
$$(1+\r)\det(I+E)=1,$$ which implies
\begin{equation}\label{00}
(1+\r)\left(1+\tr E+\f{1}{2}\left[(\tr E)^2-\tr(E^2)\right]+\det
E\right)=1
\end{equation}
since for any $3\times 3$ matrix $E$, we have $$\det(I+E)=1+\tr
E+\f{1}{2}\left[(\tr E)^2-\tr(E^2)\right]+\det E.$$ The identity
\eqref{00} further implies
\begin{equation}\label{Dv}
 \tr E=-\r-\r\tr E+(1+\r)\left[\f{1}{2}\left[\tr(E^2)-(\tr E)^2\right]-\det E\right].
\end{equation}
It is worthy to pointing out that by a similar argument, the
constraint in the two-dimensional case between the perturbations of
density and the deformation gradient has the form of
$$\tr E=-\r-\r\tr E-(1+\r)\det E.$$

Similar to the conservation of mass due to the continuity equation
\eqref{e1e1}, the quantity $$\int_\O \varrho \F_i dx$$ is also
conserved, where $\F_i$ is the i-th column of the matrix $\F$.
\begin{Proposition}\label{p2}
 If $\int_\O \varrho_0 (\F_0)_i dx=0$ and $\Dv(\varrho_0\F^\top_0)=0$, then for all time $t>0$,
\begin{equation}\label{po}
\int_\O \varrho \F_i (x,t) dx=0.
\end{equation}
\end{Proposition}

\begin{proof}
Indeed, from \eqref{e1e1} and \eqref{e1e3}, we deduce that
\begin{equation}\label{a1}
\partial_t(\varrho \F)+\u\cdot\nabla(\varrho \F)=\nabla\u(\varrho\F)-\varrho\F\Dv\u.
\end{equation}

On the other hand, from the vector identity
$$\textrm{curl}(A\times B)=A\Dv B-B\Dv A+(B\cdot\nabla)A-(A\cdot\nabla)B$$
for all vector-valued functions A, B, and $\Dv(\varrho \F^\top)=0$, one
has
\begin{equation*}
 \begin{split}
&\u\cdot\nabla(\varrho\F_i)-\nabla\u \varrho\F_i+\varrho\F_i\Dv\u\\
&\quad=\u\cdot\nabla(\varrho\F_i)-\varrho\F_i\cdot\nabla\u +\varrho\F_i\Dv\u\\
&\quad=-\nabla\times(\u\times \varrho\F_i).
 \end{split}
\end{equation*}
Thus, we can rewrite \eqref{a1} as
$$\partial_t(\varrho \F_i)-\nabla\times(\u\times \varrho\F_i)=0.$$
Integrating this identity over $\O$ gives
\begin{equation}\label{RF}
\f{d}{dt}\int_\O\varrho\F_i dx=0.
\end{equation}
The proof is complete.
\end{proof}

\bigskip

\section{Main Results}

In this paper, the standard notations for Sobolev spaces $W^{s,
q}$ and Besov spaces $B^s_{pq}$ will be used, and the following interpolation spaces will be
needed:
$$X^{2(1-\f{1}{p})}_p=\left(L^q(\O), W^{2,q}(\O)\right)_{1-\f{1}{p},p}=B^{2(1-\f{1}{p})}_{qp},$$
and
$$Y^{1-\f{1}{p}}_p=\left(L^q(\O), W^{1,q}(\O)\right)_{1-\f{1}{p},p}=B^{1-\f{1}{p}}_{qp}.$$
Now we introduce the following functional spaces to which the
solution and initial condition  of the system \eqref{e1} will
belong. Given $1\le p, q<\infty$ and $T>0$, we set
$Q_T=\O\times(0,T)$, and
$$\mathcal{W}^{p,q}(0,T):=\left\{\u\in W^{1,p}(0,T; (L^q(\O))^3)\cap L^p(0,T;(W^{2,q}(\O))^3)\right\}$$
with the norm
$$\|\u\|_{\mathcal{W}^{p,q}(0,T)}:=\|\u_t\|_{L^{p}(0,T; L^q(\O))}+\|\u\|_{L^p(0,T; W^{2,q}(\O))},$$
as well as
$$V_0^{p,q}:=\left(X^{2(1-\f{1}{p})}_p\cap Y^{1-\f{1}{p}}_p\right)^3\times \left(W^{1,q}(\O)\right)^{10}$$
with the norm
$$\|(f,g)\|_{V_0^{p,q}}:=\|f\|_{X^{2(1-\f{1}{p})}_p}+\|f\|_{Y^{1-\f{1}{p}}_p}+\|g\|_{W^{1,q}(\O)}.$$
For simplicity of notations, we drop the superscripts $p,q$ in $\mathcal{W}^{p,q}$ and
$V_0^{p,q}$, that is, we denote
$$\mathcal{W}:=\mathcal{W}^{p,q}, \qquad V_0:=V_0^{p,q}.$$

For large initial data, the following local in time well-posedness can be obtained as in \cite{HW}.

\begin{Theorem} \label{T20}
Assume that $\O$ is a bounded domain in $\R^3$ with $C^{2+\beta}$ ($\beta>0$) boundary
and $(\u_0, \r_0, E_0)$  $\in  V_0$ with
$p\in [2,\infty), q\in(3,\infty)$. There exists a positive
constant $T_0$ such that the initial-boundary value problem \eqref{e1}-\eqref{IC} has a unique strong solution on
$\O\times (0,T_0)$, satisfying
$$(\u, \r, E)\in\mathcal{W}(0,T_0)\times\left(W^{1,p}(0,T_0; L^q(\O))\cap L^p(0,T_0; W^{1,q}(\O))\right)^{10}.$$
\end{Theorem}

The argument for proving Theorem \ref{T20} is similar to that in \cite{HW}, thus we  omit the proof here.

\medskip

\begin{Remark}
An interesting case is the case $q\le p$. Indeed, by the real
interpolation method, we have
$$W^{2(1-\f{1}{p}),q}\subset B^{2(1-\f{1}{p})}_{qp}=X^{2(1-\f{1}{p})}_p,$$
and
$$W^{1-\f{1}{p},q}\subset B^{1-\f{1}{p}}_{qp}=Z^{1-\f{1}{p}}_p.$$
Then, if we replace the functional space $V_0^{p,q}$ in Theorem
\ref{T20} by
$$\mathcal{V}_0^{p,q}:=\left((W^{2(1-\f{1}{p}),q}(\O))^3\cap (W^{1-\f{1}{p}, q}(\O))^3\right)\times
(W^{1,q}(\O))^{10},$$ Theorem \ref{T20} is still valid.
\end{Remark}

Now our main result can be stated as follows.

\begin{Theorem}\label{T1}
Assume that $\O$ is a bounded domain in $\R^3$ with  $C^{2+\beta}$ ($\beta>0$) boundary and $(\u_0, \r_0, E_0)$  $\in  V_0$ with $p\in [2,\infty), q\in(3,\infty)$. There exists a  positive number $R<1$ such that if the initial data satisfies \eqref{Dv1}, \eqref{curl}, \eqref{po}, and
$$\|(\u_0, \r_0, E_0)\|_{V_0}\le R^2,$$
then the initial-boundary value problem \eqref{e1}-\eqref{IC} has a unique global strong solution 
$$(\u, \r, E)\in\mathcal{W}(0,\infty)\times\left(W^{1,p}(0,\infty; L^q(\O))\cap L^p(0,\infty; W^{1,q}(\O))\right)^{10},$$
satisfying
\begin{equation}
 \begin{cases}
  \|\u\|_{\mathcal{W}(0,\infty)}<R;\\
\|\r\|_{L^\infty(0,\infty; W^{1,q}(\O))}<R;\\
\|E\|_{L^\infty(0,\infty; W^{1,q}(\O))}<R.
 \end{cases}
\end{equation}
\end{Theorem}

\medskip

\iffalse
For Cauchy problem, we can also have the estimates for the linearized system provided that we replace $W^{1,q}$ by $W^{1,q}\cap W^{1,2}$ as $q>3$. But, unfortunately, we can not claim a similar result as Theorem \ref{T1}, because in this case, Poincare's inequality can not applied.
\fi

The rest of this paper is devoted to the proof of Theorem \ref{T1}. %This will be finished in two steps in Section 4 and Section 5.
For  simplicity of presentation, we will assume that $$|\O|=1,$$
\begin{equation}\label{a2}
\int_\O\r_0dx=0,
\end{equation}
and
\begin{equation}\label{a3}
\int_\O (1+\r_0)(I+E_0)_{ij}dx=\dl_{ij}=\begin{cases}
                                         0,\textrm{ if}\quad i\neq j;\\
1,\textrm{ if}\quad i=j.
                                        \end{cases}
\end{equation}
Note that from the continuity equation and \eqref{RF} in
Section 2, the identities \eqref{a2} and \eqref{a3}  hold also for
all positive $t> 0$.

\bigskip

%%%%%%%%%%%%%%%%%%%%%%%
\section{Estimates for the Linearized Momentum Equation} \label{Linearization}

In this section, we consider the following linearized equation:
\begin{equation}\label{l1}
 \begin{cases}
  \partial_t f-\mu\D f-(\mu+\lambda)\nabla\Dv f+\nabla h=g\\
f(0)=f_0,\quad f|_{\partial\O}=0,
 \end{cases}
\end{equation}
where $f\in \mathcal{W}(0,T)$, $g\in \left(L^p(0,T;
L^q(\O))\right)^3$, $h\in L^p(0,T; L^q(\O))$, and $f_0\in
V_0^{p,q}(\O)$. For this equation, we have

\begin{Lemma}\label{L1}
Assume that $\O$ is a bounded domain in $\R^3$ with $C^{2+\beta}$ ($\beta>0$) boundary. Then there is a positive constant $C$,
independent of $T$, such that
\begin{equation}\label{l2}
 \|f\|_{\mathcal{W}(0,T)}+\|\nabla h\|_{L^p(0,T; L^q(\O))}\le C\left(\|g\|_{L^p(0,T; L^q(\O))}+\|f_0\|_{V_0^{p,q}(\O)}\right).
\end{equation}
\end{Lemma}

The strategy of proving Lemma \ref{L1} is to decompose the
function $f$ into two parts: one is a solution to Stokes equations
and the other is a solution to a parabolic equation. For this
purpose, let us first recall the maximal regularity theorem for
parabolic equations (cf. \cite{AI}).

\begin{Proposition}\label{T3}
Assume that $\O$ is a bounded domain in $\R^3$ with $C^{2+\beta}$ ($\beta>0$) boundary.
Given $1<p<\infty$, $\omega_0\in V_0^{p,q}$ and $\varphi\in L^p(0,T;
L^q(\O))$, the Cauchy problem
\begin{equation*}
 \begin{cases}
\f{\partial\omega}{\partial t}-\mu\D\omega-(\mu+\lambda)\nabla\Dv\omega=\varphi,\\
\omega(0)=\omega_0,
\end{cases}
\end{equation*}
has a unique solution $\omega\in \mathcal{W}(0,T)$, and
$$\|\omega\|_{\mathcal{W}(0,T)}\le C_1\left(\|\varphi\|_{L^p(0,T; L^q(\O))}+\|\omega_0\|_{V_0^{p,q}}\right),$$
where $C_1$ is independent of $\omega_0$ and $\varphi$. In addition,
there exists a positive constant $c_0$ independent of $\varphi$ such
that
$$\|\omega\|_{\mathcal{W}(0,T)}\ge c_0\sup_{t\in(0,T)}\|\omega(t)\|_{V_0^{p,q}}.$$
\end{Proposition}

The classical $L^p-L^q$ estimates for Stokes equations can be stated as follows (see
Theorem 3.2 in \cite{DR}).

\begin{Proposition}\label{SE}
Assume that $\O$ is a bounded domain in $\R^3$ with $C^{2+\beta}$ ($\beta>0$) boundary, and
$\psi\in L^s(\R^+, L^q(\O))$ for $1<q,s<\infty$. Then the initial-boundary value problem of the Stokes equations
\begin{equation*}
 \begin{cases}
  \partial_t\u-\mu\D\u+\nabla\pi=\psi,\quad \int_{\O}\pi dx=0,\\
\Dv\u=0,\\
\u|_{\partial\O}=\u|_{t=0}=0,
 \end{cases}
\end{equation*}
has a unique solution $(\u,\pi)$ satisfying the following inequality for all $T>0$:
\begin{equation*}
 \begin{split}
  \|(\D\u,\nabla\pi,\u_t)\|_{L^s(0,T; L^q(\O))}\le C\|\psi\|_{L^s(0,T; L^q(\O))}
 \end{split}
\end{equation*}
with $C=C(q,s,\O,\mu)$.
\end{Proposition}

\begin{proof}[Proof of Lemma \ref{L1}]
To begin with, we first introduce two new functions $f_1$ and $f_2$, which satisfy respectively
\begin{equation}\label{l3}
 \begin{cases}
   \partial_t f_1-\mu\D f_1+\nabla h=g\\
\Dv f_1=0\\
f_1(0)=0,\quad f_1|_{\partial\O}=0,
 \end{cases}
\end{equation}
and
\begin{equation}\label{l4}
 \begin{cases}
   \partial_t f_2-\mu\D f_2-(\mu+\lambda)\nabla\Dv f_2=0\\
f_2(0)=f_0,\quad f_2|_{\partial\O}=0.
 \end{cases}
\end{equation}
Notice that $$f=f_1+f_2$$ by the linearity of the equation
\eqref{l1}. For Stokes system \eqref{l3}, using Proposition
\ref{SE}, there exists a solution $f_1$ to \eqref{l3} and it
satisfies
$$\|f_1\|_{\mathcal{W}(0,T)}+\|\nabla h\|_{L^p(0,T; L^q(\O))}\le C\|g\|_{L^p(0,T; L^q(\O))},$$
where the positive constant $C$ depends on $p,q, \O$ and does not depend on $T$.
On the other hand, by the maximal regularity theorem of parabolic equations, there exists a solution $f_2$ 
(extended to zero outside $\O$) to \eqref{l4} and $f_2$ satisfies
$$\|f_2\|_{\mathcal{W}(0,T)}\le C\|f_0\|_{V_0^{p,q}},$$
where the positive constant $C$ depends on $\mu, \lambda, p, q$ and does not depend on $T$.
Hence, we have
\begin{equation*}
\begin{split}
 \|f\|_{\mathcal{W}(0,T)}+\|\nabla h\|_{L^p(0,T; L^q(\O))}&\le \|f_1\|_{\mathcal{W}(0,T)}+\|f_2\|_{\mathcal{W}(0,T)}+\|\nabla h\|_{L^p(0,T; L^q(\O))}\\
&\le C\left(\|g\|_{L^p(0,T; L^q(\O))}+\|f_0\|_{V_0^{p,q}}\right),
\end{split}
\end{equation*}
where the positive constant $C$ depends on $p,q, \O$ and does not depend on $T$.

\end{proof}

\bigskip

%%%%%%%%%%%%%%%%%%%%%%%
\section{Uniform Estimates} \label{estimates}

Due to Theorem \ref{T20}, for any given initial data, the local existence and uniqueness of strong solution can be established. In order to extend the local solution to a global one, we need to establish a series of uniform bounds. That is the main objective of this section.

Throughout this section, we assume that over the time interval
$[0,T]$, for a given sufficiently small positive number $R<1$, the
following bounds hold
\begin{equation}\label{a4}
 \begin{cases}
  \|\u\|_{\mathcal{W}(0,T)}\le R;\\
\|\r\|_{L^\infty(0,T; W^{1,q}(\O))}\le R;\\
\|E\|_{L^\infty(0,T; W^{1,q}(\O))}\le R.
 \end{cases}
\end{equation}
The main goal of this section is to obtain some uniform bounds on $\nabla\r$ and $\nabla E$ in $L^p(0,T; L^q(\O))$ (see Corollary \ref{c1} below).

To begin with, notice that since $q>3$, \eqref{a4} will imply
$$\|\r\|_{L^\infty(\O)}\le C\|\r\|_{W^{1,q}(\O)}\le CR<\f{1}{2},$$
if $R$ is sufficiently small.
Similarly, one can assume
$$\|E_{ij}\|_{L^\infty}\le CR\le 1,\quad\textrm{for all}\quad i,j=1,2,3.$$

\subsection{Dissipation of the gradient of the density}
To prove Theorem \ref{T1} valid, we need first to establish a uniform
estimate on the dissipation of the gradient of the density.

\begin{Lemma}\label{E2}
Under the same condition as Theorem \ref{T1}, the solution $(\r,\u, E)$ satisfies
\begin{equation}\label{12345}
\|\nabla\r\|_{L^p(0,T; L^q(\O))}\le C(R^2+R\|\nabla E\|_{L^p(0,T;
L^q(\O))}),
\end{equation}
where $C=C(p,q,\mu,\lambda,  \O).$
\end{Lemma}

\begin{proof}
To begin with, we introduce the zero-th  order differential operator
$$\mathcal{R}_{ij}=\D^{-1}\partial_{x_j}\partial_{x_i},$$ which is
defined as
$$\mathcal{R}_{ij}f=-\mathcal{F}^{-1}\left(\f{\xi_i\xi_j}{|\xi|^2}\mathcal{F}f\right),$$
where $\mathcal{F}$ denotes the Fourier transformation. Note that
$$\|\mathcal{R}_{ij}f\|_{L^q}\le C\|f\|_{L^q},\quad\textrm{for}\quad 1<q<\infty\quad\textrm{and}\quad 1\le i,j\le 3.$$ 
Extending the perturbations of the density and the deformation gradient by $0$ outside the domain $\O$, applying
the operator $\mathcal{R}_{ij}$ to \eqref{e1e2}, and using the equation \eqref{e1e1}, we obtain
\begin{equation}\label{x1}
\begin{split}
&\mathcal{R}_{ij}\left((1+\r)\u_t+(1+\r)\u\cdot\nabla\u\right)-(2\mu+\lambda)\partial_{x_j}\Dv
\u+\alpha\partial_{x_j} \r\\&\quad=\mathcal{R}_{ij}\left(\r\tr
E+(1+\r)E_{lk}\partial_{x_l}E_{ik}\right)+\mathcal{M}_j+\left(P'(1)-P'(1+\r)\right)\partial_{x_j}\r,
\end{split}
\end{equation}
where $$\alpha=(1+P'(1)),$$
and we used the fact that, from \eqref{curl1} and \eqref{Dv},
\begin{equation}\label{m}
\begin{split}
\mathcal{R}_{ij}\Dv
E&=\D^{-1}\partial_{x_j}\partial_{x_i}\partial_{x_k}E_{ik}
=\D^{-1}\partial_{x_j}\partial_{x_k}\partial_{x_i}E_{ik}\\
&%\overset{\eqref{curl1}}
=\D^{-1}\partial_{x_j}\partial_{x_k}\partial_{x_k}E_{ii}+\mathcal{R}_{jk}\left(E_{li}\nabla_l
E_{ik}-E_{lk}\nabla_l E_{ii}\right)\\
&=\partial_{x_j}\tr E+\mathcal{R}_{jk}\left(E_{li}\nabla_l
E_{ik}-E_{lk}\nabla_l E_{ii}\right)\\
&%\overset{\eqref{Dv}}
=-\partial_{x_j}\r-\partial_{x_j}(\r\tr E)+\partial_{x_j}\left((1+\r)\left[\f{1}{2}\left[\tr(E^2)-(\tr E)^2\right]-\det E\right]\right)\\
&\qquad+\mathcal{R}_{jk}\left(E_{li}\nabla_l
E_{ik}-E_{lk}\nabla_l E_{ii}\right)\\
&=:-\partial_{x_j}\r+\mathcal{M}_j,
\end{split}
\end{equation}
with
\begin{equation*}
 \begin{split}
\mathcal{M}_j=\partial_{x_j}\left((1+\r)\left[\f{1}{2}\left[\tr(E^2)-(\tr
E)^2\right]-\det
E\right]\right)+\mathcal{R}_{jk}\left(E_{li}\nabla_l
E_{ik}-E_{lk}\nabla_l E_{ii}\right).
 \end{split}
\end{equation*}

We rewrite the equation \eqref{x1} as
\begin{equation}\label{x2}
\begin{split}
&(\mathcal{R}_{ij}\u)_t-\mu\D\mathcal{R}_{ij}\u-(\mu+\lambda)\nabla\Dv
\mathcal{R}_{ij}\u+\alpha\partial_{x_j}
\r\\&\quad=-\mathcal{R}_{ij}\left(\r\u_t+(1+\r)\u\cdot\nabla\u\right)+\mathcal{R}_{ij}\left(\r\tr
E+(1+\r)E_{lk}\partial_{x_l}E_{ik}\right)+\mathcal{M}_j\\
&\qquad+\left(P'(1)-P'(1+\r)\right)\partial_{x_j}\r.
\end{split}
\end{equation}
According to Lemma \ref{L1}, we have
\begin{equation}\label{x3}
 \begin{split}
&  \|\mathcal{R}_{ij}\u\|_{\mathcal{W}(0,T)}+\alpha\|\nabla \r\|_{L^p(0,T; L^q(\O))}\\
&\le C\Big(\|\mathcal{R}_{ij}\u_0\|_{V_0^{p,q}}+\|\mathcal{R}_{ij}\left(\r\u_t+(1+\r)\u\cdot\nabla\u\right)\|_{L^p(0,T; L^q(\O))}\\
&\quad+\|\mathcal{R}_{ij}\left(\r\tr E+(1+\r)E_{lk}\partial_{x_l}E_{ik}\right)\|_{L^p(0,T;L^q(\O))}\\
&\quad+\|\mathcal{M}_j\|_{L^p(0,T;
L^q(\O))}+\left\|\left(P'(1)-P'(1+\r)\right)\partial_{x_j}\r\right\|_{L^p(0,T;
L^q(\O))}\Big).
 \end{split}
\end{equation}

Next, we  estimate the terms on the right hand side of the above inequality. Indeed,
\begin{equation*}
 \begin{split}
  &\|\mathcal{R}_{ij}\left(\r\u_t+(1+\r)\u\cdot\nabla\u\right)\|_{L^p(0,T; L^q(\O))}\\&\le\|\r\u_t+(1+\r)\u\cdot\nabla\u\|_{L^p(0,T; L^q(\O))}\\
&\le \|\r\|_{L^\infty}\|\u_t\|_{L^p(0,T; L^q(\O))}+(1+\|\r\|_{\infty})\|\u\|_{L^\infty(0,T; L^q(\O))}\|\nabla\u\|_{L^p(0,T; L^\infty(\O))}\\
&\le C\Big(\|\r\|_{L^\infty}\|\u_t\|_{L^p(0,T; L^q(\O))}+(1+\|\r\|_{\infty})\|\u\|_{L^\infty(0,T; L^q(\O))}\|\u\|_{L^p(0,T; W^{2,q}(\O))}\Big)\\
&\le CR^2;
 \end{split}
\end{equation*}

\begin{equation*}
 \begin{split}
&\|\mathcal{R}_{ij}\left(\r\tr E+(1+\r)E_{lk}\partial_{x_l}E_{ik}\right)\|_{L^p(0,T;L^q(\O))}\\&\le\|\r\tr E+(1+\r)E_{lk}\partial_{x_l}E_{ik}\|_{L^p(0,T;L^q(\O))}\\
&\le \|E\|_{L^\infty}\|\r\|_{L^p(0,T; L^q(\O))}+(1+\|\r\|_{L^\infty})\|E\|_{L^\infty}\|\nabla E\|_{L^p(0,T; L^q(\O))}\\
&\le C\left(R^2+R\|\nabla E\|_{L^p(0,T; L^q(\O))}\right);
 \end{split}
\end{equation*}

\begin{equation}\label{m1}
 \begin{split}
\|\mathcal{M}_j\|_{L^p(0,T; L^q(\O))}&\le
\|\mathcal{R}\left(E_{li}\nabla_l
E_{ik}-E_{lk}\nabla_l E_{ii}\right)\|_{L^p(0,T;L^q(\O))}\\
&\quad+\left\|\partial_{x_j}\left((1+\r)\left[\f{1}{2}\left[\tr(E^2)-(\tr E)^2\right]-\det E\right]\right)\right\|_{L^p(0,T; L^q(\O))}\\
& \le C\Big(\|E_{li}\nabla_l
E_{ik}-E_{lk}\nabla_l E_{ii}\|_{L^p(0,T;L^q(\O))}\\
&\quad+\|\nabla\r\|_{L^p(0,T; L^q(\O))}(\|E\|_{L^\infty}^2+\|E\|_{L^\infty}^3)\\
&\quad+(1+\|\r\|_{L^\infty})\|\nabla E\|_{L^p(0,T; L^q(\O))}(\|E\|_{L^\infty}+\|E\|_{L^\infty}^2)\Big)\\
&\le C\Big(R\|\nabla E\|_{L^p(0,T; L^q(\O))}+R\|\nabla\r\|_{L^p(0,T; L^q(\O))}\Big),
\end{split}
\end{equation}
since $R<1$;

\begin{equation*}
\begin{split}
\left\|\left(P'(1)-P'(1+\r)\right)\partial_{x_j}\r\right\|_{L^p(0,T;
L^q(\O))}&\le \eta\left\|\partial_{x_j}\r\right\|_{L^p(0,T;
L^q(\O))}\|\r\|_{L^\infty}\\
&\le CR\|\nabla\r\|_{L^p(0,T; L^q(\O))},
\end{split}
\end{equation*}
here we used the fact
$$P'(1+\r)-P'(1)=P''(z)\r\quad\textrm{for some $z$ between 1 and
$1+\r$},$$ and hence
\begin{equation*}
\begin{split}
\|P'(1+\r)-P'(1)\|_{L^\infty}\le \eta\|\r\|_{L^\infty}
\end{split}
\end{equation*}
with
$$\eta=\sup_{\f12\le z\le \f32}|P''(z)|.$$

Substituting those estimates back into \eqref{x3}, we obtain
\begin{equation*}
 \begin{split}
  \|\mathcal{R}_{ij}\u\|_{\mathcal{W}(0,T)}&+\alpha\|\nabla \r\|_{L^p(0,T; L^q(\O))}\\
  &\le C\Big(R^2+R\|\nabla E\|_{L^p(0,T; L^q(\O))}+R\|\nabla\r\|_{L^p(0,T; L^q(\O))}\Big).
 \end{split}
\end{equation*}
Assuming that $R$ is so small that $CR<\f{\alpha}{2}$, we deduce from the above inequality that
\begin{equation}\label{x4}
 \begin{split}
  \|\mathcal{R}_{ij}\u\|_{\mathcal{W}(0,T)}+\f{\alpha}{2}\|\nabla \r\|_{L^p(0,T; L^q(\O))}&\le C\Big(R^2+R\|\nabla E\|_{L^p(0,T; L^q(\O))}\Big).
 \end{split}
\end{equation}
The proof is complete.
\end{proof}

\subsection{Dissipation of the deformation gradient}
The main difficulty of the proof of Theorem \ref{T1} is to obtain
estimates on the dissipation of the deformation gradient. This is
partly because of the transport structure of equation \eqref{e13}.
It is worthy of pointing out that it is extremely difficult to
directly deduce the dissipation of the deformation gradient.
Fortunately, for the viscoelastic fluids system \eqref{e1}, as we
can see in \cite{CZ, LLZH2, LLZH, LZ, LLZ, LZP}, some sort
of combinations between the gradient of the velocity and the
deformation gradient indeed induce some dissipation. To make this
statement more precise, we rewrite the momentum equation
\eqref{e12} as, using \eqref{Dv1}
\begin{equation}\label{511}
\begin{split}
\partial_t \u-\mu\Delta \u-\Dv E&=-(1+\r)(\u\cdot\nabla)\u-\nabla P(1+\r)\\
&\quad+\r\Dv E+(1+\r)E_{jk}\partial_{x_j}E_{ik}-\r\partial_t\u,
\end{split}
\end{equation}
and prove the following estimate:

\begin{Lemma}\label{E1}
Under the same condition as Theorem \ref{T1}, the solution $(\r,\u, E)$ satisfies
\begin{equation}\label{5113}
\|\nabla E\|_{L^p(0,T; L^q(\O))} \le CR,
\end{equation}
where $C=C(p,q,\mu,\lambda)$.
\end{Lemma}

\begin{proof}
Now we introduce the function $Z_1(x,t)=\mathcal{L}(\Dv E)$ as
\begin{equation}\label{512}
\begin{cases}
-\mu\D \mathcal{L}(f)-(\mu+\lambda)\nabla\Dv\mathcal{L}(f)=f,\\
\mathcal{L}(f)|_{\partial\O}=0.
\end{cases}
\end{equation}
Then, \eqref{511} becomes
\begin{equation}\label{513}
\partial_t \u-\mu\Delta\left(\u-\f{1}{\mu}Z_1\right)=\mathcal{H}_1,
\end{equation}
where
\begin{equation*}
\begin{split}
\mathcal{H}_1=-(1+\r)(\u\cdot\nabla)\u-\nabla
P(1+\r)+\r\Dv E+(1+\r)E_{jk}\partial_{x_j}E_{ik}-\r\partial_t\u.
\end{split}
\end{equation*}
Also, from \eqref{e13}, we have
\begin{equation}\label{514}
\begin{split}
\f{\partial Z_1}{\partial t}&=\mathcal{L}\left(\Dv \f{\partial E}{\partial t}\right)\\
&=\mathcal{L}\Big(\Dv (\nabla \u+\nabla \u E-(\u\cdot\nabla)E)\Big).
\end{split}
\end{equation}
From \eqref{513} and \eqref{514}, we deduce, denoting
$Z=\u-\f{1}{\mu}Z_1$,
\begin{equation}\label{515}
\begin{split}
\partial_t Z-\mu\Delta Z-(\mu+\lambda)\nabla\Dv Z=\mathcal{H}:=\mathcal{H}_1-\mathcal{H}_2,
\end{split}
\end{equation}
where
$$\mathcal{H}_2=\f{1}{\mu}\mathcal{L}(\Dv (\nabla \u+\nabla \u E-(\u\cdot\nabla)E)).$$
Equation \eqref{515} with Proposition \ref{T3} implies that
\begin{equation}\label{516}
\begin{split}
\|Z\|_{\mathcal{W}(0,T)}&\le C(p,q)\left(\|Z(0)\|_{X_p^{2(1-\f{1}{p})}}+\|\mathcal{H}\|_{L^p(0,T;L^q(\O))}\right)\\
&\le C(p,q)\left(R+\|\mathcal{H}\|_{L^p(0,T;L^q(\O))}\right).
\end{split}
\end{equation}

Next, we estimate $\|\mathcal{H}_i\|_{L^p(0,T;L^q(\O))}$, $i=1,2$. Indeed, for $\mathcal{H}_1$, using \eqref{12345}, we
have
\begin{equation}\label{517}
\begin{split}
\|\mathcal{H}_1\|_{L^p(0,T; L^q(\O))}&\le (1+\|\r\|_{L^\infty(Q_T)})\|\u\|_{L^\infty(0,T; L^q(\O))}\|\nabla \u\|_{L^p(0,T; L^\infty(\O))}\\
&\quad+\kappa\|\nabla\r\|_{L^p(0,T; L^q(\O))}+\|\r\|_{L^\infty(Q_T)}\|\nabla E\|_{L^p(0,T; L^q(\O))}\\
&\quad+(1+\|\r\|_{L^\infty(Q_T)})\|E\|_{L^\infty(Q_T)}\|\nabla E\|_{L^p(0,T; L^q(\O))}\\
&\quad+\|\r\|_{L^\infty(Q_T)}\|\partial_t\u\|_{L^p(0,T; L^q(\O))}\\
&\le C\Big(R+R\|\u\|_{L^p(0,T; W^{2,q}(\O))}+R\|\nabla E\|_{L^p(0,T; L^q(\O))}\Big)\\
&\le C\Big(R+R\|\nabla E\|_{L^p(0,T; L^q(\O))}\Big),
\end{split}
\end{equation}
since $R<1$ with $$\kappa=\sup_{|z|\le\frac{1}{2}}P'(1+z).$$
For $\mathcal{H}_2$, we have
$$|\mathcal{H}_2|\le \f{1}{\mu}|\u|+\f{1}{\mu}\mathcal{L}(\Dv(\nabla \u E-(\u\cdot\nabla)E)),$$
and
\begin{equation*}
\begin{split}
\|\nabla \u E-(\u\cdot\nabla)E\|_{L^p(0,T; L^{q}(\O))}&\le \|\nabla\u\|_{L^p(0,T; L^q(\O))}\|E\|_{L^\infty}\\
&\quad+\|\u\|_{L^\infty}\|\nabla E\|_{L^p(0,T; L^q(\O))}\\
&\le R^{2}.
\end{split}
\end{equation*}
Hence, one can estimate, by the standard estimates of elliptic equations,
\begin{equation}\label{518}
\begin{split}
\|\mathcal{H}_2\|_{L^p(0,T; L^{q}(\O))}&\le
\f{1}{\mu}\|\u\|_{L^p(0,T;L^q(\O))}+\f{1}{\mu}\|\nabla \u E-(\u\cdot\nabla)E\|_{L^p(0,T; L^{q}(\O))}\\
&\le C(R+R^{2})\le CR.
\end{split}
\end{equation}
Therefore, from \eqref{517} and \eqref{518}, we obtain
\begin{equation}\label{519}
\begin{split}
\|\mathcal{H}\|_{L^p(0,T; L^q(\O))}&\le
C\Big(R+R\|\nabla E\|_{L^p(0,T;
L^q(\O))}\Big).
\end{split}
\end{equation}
Inequalities \eqref{516} and \eqref{519} imply that
\begin{equation}\label{5110}
\begin{split}
\|Z\|_{L^p(0,T; W^{2,q}(\O))}\le
C\Big(R+R\|\nabla E\|_{L^p(0,T;
L^q(\O))}\Big).
\end{split}
\end{equation}
Hence, we have, from \eqref{512}
\begin{equation}\label{5111}
\begin{split}
\|\Dv E\|_{L^p(0,T; L^q(\O))}&\le \mu\left(\|Z\|_{L^p(0,T; W^{2,q}(\O))}+\|\u\|_{L^p(0,T; W^{2,q}(\O))}\right)\\
&\le C\Big(R+R\|\nabla E\|_{L^p(0,T;
L^q(\O))}\Big)).
\end{split}
\end{equation}
On the other hand, from the identity \eqref{curl1}, we deduce that
\begin{equation}\label{5112}
\begin{split}
\|\textrm{curl }E_i\|_{L^p(0,T; L^q(\O))}&\le 2\|E\|_{L^\infty(Q_T)}\|\nabla E\|_{L^p(0,T; L^q(\O))}\\
&\le C\|E\|_{L^\infty(0,T; W^{1,q}(\O))}\|\nabla E\|_{L^p(0,T;
L^q(\O))}\\&\le CR\|\nabla E\|_{L^p(0,T; L^q(\O))}.
\end{split}
\end{equation}
Combining together \eqref{5111}, \eqref{5112} and the inequality
$$\|\nabla f\|_{L^r}\le C(\|\Dv f\|_{L^r}+\|\textrm{curl}f\|_{L^r}) \textrm{ for all }1<r<\infty,$$ we obtain
\begin{equation*}
\begin{split}
\|\nabla E\|_{L^p(0,T; L^q(\O))} &\le
C\Big(R+R\|\nabla E\|_{L^p(0,T;
L^q(\O))}\Big),
\end{split}
\end{equation*}
and hence, by choosing $CR\le \f{1}{2}$,
one obtains \eqref{5113}.

The proof of Lemma \ref{E1} is complete.
\end{proof}

Notice that,  from \eqref{12345} and \eqref{5113}, one can improve the estimates on the gradient of the density and the divergence of the deformation gradient. Indeed, we have
\begin{Corollary}\label{c1}
Under the same condition as Theorem \ref{T1}, the solution $(\r,\u, E)$ satisfies
\begin{equation}\label{2222}
\|\nabla \r\|_{L^p(0,T; L^{q}(\O))}\le CR^2,
\end{equation}
and
\begin{equation}\label{22222}
 \|\nabla E\|_{L^p(0,T; L^q(\O))}\le CR^2,
\end{equation}
where the positive constant $C$ is independent of $T$.
\end{Corollary}

\begin{proof}
First, note that the estimate \eqref{2222} is a direct consequence of Lemma \ref{E2} and Lemma \ref{E1}.

To show \eqref{22222}, according to \eqref{5112} and Lemma \ref{E1}, we only need to verify
$$\|\Dv E\|_{L^p(0,T; L^q(\O))}\le CR^2.$$  To this end,  we notice first that, since $\Dv\,\textrm{curl}f=0$ for any vector-valued function $f$, thus
$$(-\D)^{-1}\Dv\,\textrm{curl}\,\Dv E=0.$$
On the other hand, we also have, using the identity \eqref{curl1},
\begin{equation*}
 \begin{split}
 & (-\D)^{-1}\Dv\,\textrm{curl}\,\Dv E
 =(-\D)^{-1}\Dv(\partial_{x_k}\partial_{x_j}E_{ij}-\partial_{x_i}\partial_{x_j}E_{kj})\\
&=(-\D)^{-1}\Dv(\partial_{x_j}\partial_{x_j}E_{ik}-\partial_{x_j}\partial_{x_j}E_{ki})\\
&\quad+(-\D)^{-1}\Dv\partial_{x_j}(E_{lj}\nabla_lE_{ik}-E_lk\nabla_l E_{ij}-E_{lj}\nabla_lE_{ki}+E_li\nabla_l E_{kj})\\
&=\Dv E-\Dv E^\top+\mathcal{N}
 \end{split}
\end{equation*}
with
$$\mathcal{N}=(-\D)^{-1}\Dv\partial_{x_j}(E_{lj}\nabla_lE_{ik}-E_{lk}\nabla_l E_{ij}-E_{lj}\nabla_lE_{ki}+E_{li}\nabla_l E_{kj}).$$
Hence, we have
\begin{equation}\label{q1}
\Dv E=\Dv E^\top-\mathcal{N}.
\end{equation}

The property \eqref{Dv1} implies that
$$\Dv E^\top=-\nabla\r-\Dv(\r E^\top), $$
which together with \eqref{2222} yield
\begin{equation}\label{q2}
\begin{split}
 \|\Dv E^\top\|_{L^p(0,T; L^q(\O))}&=\|\nabla\r+\Dv(\r E^\top)\|_{L^p(0,T; L^q(\O))}\\
&\le \|\nabla\r\|_{L^p(0,T; L^q(\O))}+\|\r\|_{L^\infty}\|\nabla E\|_{L^p(0,T; L^q(\O))}\\
&\quad+\|E\|_{L^\infty}\|\nabla\r\|_{L^p(0,T; L^q(\O))}\\
&\le CR^2.
\end{split}
\end{equation}
For $\mathcal{N}$, we have, using Lemma \ref{E1},
\begin{equation}\label{q3}
 \begin{split}
  \|\mathcal{N}\|_{L^p(0,T; L^q(\O))}&\le 4\|E\|_{L^\infty}\|\nabla E\|_{L^p(0,T; L^q(\O))}\\
&\le CR^2
 \end{split}
\end{equation}

Combining \eqref{q1}, \eqref{q2}, and \eqref{q3}, we have
$$\|\Dv E\|_{L^p(0,T; L^q(\O))}\le CR^2.$$
The proof is complete.
\end{proof}

\bigskip

\section{Proof of Theorem \ref{T1}}

The goal of  this section is to prove Theorem \ref{T1}.
To this end, for a given $R>0$ small enough, we define
\begin{equation*}
\begin{split}
 T_{\textrm{max}}:=\sup\Big\{&T>0:\;  \textrm{there exists a solution $(\r,\u, E)$ to \eqref{e1} with }\|\u\|_{\mathcal{W}(0,T)}<R\textrm{ and }\\
&\max\{\|\r\|_{L^\infty(0,T; W^{1,q}(\O))}, \|E\|_{L^\infty(0,T; W^{1,q}(\O))}\}<R\Big\}.
\end{split}
\end{equation*}

We first show that in the interval $[0, T_{\textrm{max}}]$, we have better estimates on the velocity.

\begin{Lemma}\label{l51}
For any $T\in [0,T_{\textrm{max}}]$, we have
$$\|\u\|_{\mathcal{W}(0,T)}\le CR^2$$
for some positive constant $C$ independent of $T$.
\end{Lemma}
\begin{proof}
In fact, by Proposition \ref{T3} and equation \eqref{e12}, we have
\begin{equation}\label{x6}
 \begin{split}
\|\u\|_{\mathcal{W}(0,T)}
&\le C(p,q,\O)\Big(\|\u_0\|_{V_0^{p,q}}+\|\r\partial_t\u\|_{L^p(0,T;L^q(\O))}\\
&\quad+\|(1+\r)\u\cdot\nabla\u\|_{L^p(0,T; L^q(\O))}\\
&\quad+\|\nabla P(1+\r)\|_{L^p(0,T;L^q(\O))}\\
&\quad+\|(1+\r)\Dv E\|_{L^p(0,T; L^q(\O))}\\
&\quad+\|(1+\r) E_{jk}\partial_{x_j} E_{ik}\|_{L^p(0,T; L^q(\O))}\Big).
\end{split}
\end{equation}
For those terms on the right hand side of \eqref{x6}, in view of Lemma \ref{E1} and Corollary \ref{c1},  we have the following estimates:
\begin{equation*}
 \|\r\partial_t\u\|_{L^p(0,T;L^q(\O))}\le \|\r\|_{L^\infty}\|\partial_t\u\|_{L^p(0,T;L^q(\O))}\le CR^2;
\end{equation*}
\begin{equation*}
 \begin{split}
  \|(1+\r)\u\cdot\nabla\u\|_{L^p(0,T; L^q(\O))}&\le (1+\|\r\|_{L^\infty})\|\u\|_{L^\infty(0,T; L^q(\O))}\|\nabla\u\|_{L^p(0,T; L^\infty)}\\
&\le C\|\u\|_{L^\infty(0,T; L^q(\O))}\|\u\|_{L^p(0,T; W^{2,q}(\O))}\\
&\le CR^2;
 \end{split}
\end{equation*}
\begin{equation*}
 \begin{split}
  \|\nabla P(1+\r)\|_{L^p(0,T;L^q(\O))}&\le \kappa\|\nabla\r\|_{L^p(0,T; L^q(\O))}\le CR^2;
 \end{split}
\end{equation*}
\begin{equation*}
 \begin{split}
\|(1+\r)\Dv E\|_{L^p(0,T; L^q(\O))}&\le (1+\|\r\|_{L^\infty})\|\Dv E\|_{L^p(0,T; L^q(\O))}\le CR^2;
 \end{split}
\end{equation*}
\begin{equation*}
 \begin{split}
  \|(1+\r) E_{jk}\partial_{x_j} E_{ik}\|_{L^p(0,T; L^q(\O))}&\le (1+\|\r\|_{L^\infty})\|E\|_{L^\infty}\|\nabla\E\|_{L^p(0,T; L^q(\O))}\\
&\le C\|E\|_{L^\infty}\|\nabla\E\|_{L^p(0,T; L^q(\O))}\\
&\le CR^2.
 \end{split}
\end{equation*}

Substituting those estimates back into \eqref{x6}, we have
\begin{equation*}
 \|\u\|_{\mathcal{W}(0,T)}\le C(\|\u_0\|_{V_0^{p,q}}+R^2)\le CR^2,
\end{equation*}
provided that we choose the initial data to satisfy
$$\|\u_0\|_{V_0^{p,q}}\le CR^2.$$
The proof is complete.

\end{proof}

According to Lemma \ref{l51}, we deduce from the above estimate that
$$\|\u\|_{\mathcal{W}(0,T_{\textrm{max}})}\le CR^2.$$

Next, we improve the estimates on the density and the deformation gradient. Due to the similarity of equations for the density and the deformation gradient, the arguments  for those two unknowns are similar, and hence for the clarity of the presentation, we will only focus on the argument for the improved estimates for the density. To this end, we introduce a new variable:
$$\sigma:=\nabla\ln \varrho,$$
and we have

\begin{Lemma}\label{sigma}
Function $\sigma$ satisfies
\begin{equation}\label{xx2}
\partial_t \sigma+\nabla(\u\cdot\sigma)=-\nabla\Dv\u,
\end{equation}
in the sense of distributions. Moreover, the norm
$\|\sigma(t)\|_{L^q(\O)}$ is continuous in time.
\end{Lemma}

\begin{proof}
We follow the argument in \cite{AI} (Section 9.8) by denoting
$\sigma_\varepsilon=S_\varepsilon\sigma$, where $S_\varepsilon$ is
the standard mollifier in the spatial variables. Then, we have
$$\partial_t\sigma_\varepsilon+\nabla(\u\cdot\sigma_\varepsilon)=-\nabla\Dv\u_\varepsilon+\mathcal{R}_\varepsilon,$$
with
\begin{equation}\label{xxx1}
\begin{split}
\mathcal{R}_\varepsilon&=\nabla(\u\cdot\sigma_\varepsilon)-S_\varepsilon\nabla(\u\cdot\sigma)\\
&=(\u\cdot\nabla\sigma_\varepsilon-S_\varepsilon(\u\cdot\nabla\sigma))+(\sigma_\varepsilon\nabla \u-S_\varepsilon(\sigma\cdot\nabla \u))\\
&=:\mathcal{R}_\varepsilon^1+\mathcal{R}_\varepsilon^2.
\end{split}
\end{equation}
Since $\sigma\in L^\infty(0,T; L^q(\O))$ and $\u\in L^p(0,T;
W^{1,\infty}(\O))$, we deduce from Lemma 6.7 in \cite{AI} (cf.
Lemma 2.3 in \cite{LI}) that
$$\mathcal{R}_\varepsilon^1\to   0 \quad\text{as}\quad
\varepsilon\to   0.$$
  Moreover,
$$\|(\sigma_\varepsilon-\sigma)\nabla \u\|_{L^1(0,T; L^q(\O))}\le\|\sigma-\sigma_\varepsilon\|_{L^{\f{p}{p-1}}(0,T; L^q(\O))}
\|\nabla \u\|_{L^p(0,T; L^\infty(\O))}\to   0,$$ and
$$S_\varepsilon(\sigma\cdot\nabla \u)\to  \sigma\cdot\nabla \u\quad\text{in}\quad
L^1(0,T; L^q(\O))$$ since $\sigma\cdot\nabla \u\in L^p(0,T;
L^q(\O))$. Thus, we have $$\mathcal{R}_\varepsilon^2\to   0 \quad\text{in}\quad
L^p(0,T; L^q(\O)).$$ Then, taking the limit as $\varepsilon\to 0$
in \eqref{xxx1}, we get \eqref{xx2}.

Multiplying \eqref{xx2} by $|\sigma|^{q-2} \sigma$, and
integrating over $\O$, we get
\begin{equation*}
\begin{split}
\f{1}{q}\left|\f{d}{dt}\|\sigma\|^q_{L^q(\O)}\right|&=\left|\int_{\O}(-\nabla\Dv\u\cdot\sigma
|\sigma|^{q-2}-\partial_j\u_k \sigma_j\sigma_k|\sigma|^{q-2}
-\f{1}{q}\Dv \u|\sigma|^q)dx\right|\\
&\le \|\u\|_{W^{2,q}(\O)}\|\sigma\|_{L^q(\O)}^{q-1}+\|\nabla \u\|_{L^\infty}\|\sigma\|^q_{L^q}+\f{1}{q}\|\Dv \u\|_{L^\infty}\|\sigma\|_{L^q}^q\\
&\le C\|\u\|_{W^{2,q}}\|\sigma\|^{q-1}_{L^q}(1+\|\sigma\|_{L^q}).
\end{split}
\end{equation*}
Dividing the above inequality by $\|\sigma\|^{q-1}_{L^q}$, we
obtain
$$\left|\f{d}{dt}\|\sigma\|_{L^q}\right|\le C\|\u\|_{W^{2,q}}(1+\|\sigma\|_{L^q}).$$
Since $\sigma\in L^\infty(0, T; L^q(\O))$, we have
$$\f{d}{dt}\|\sigma\|_{L^q}\in L^p(0,T).$$
 Thus, $\|\sigma\|_{L^q}\in C(0,T).$
The proof of Lemma \ref{sigma} is
complete.
\end{proof}

With the aid of Lemma \ref{sigma}, we can improve the estimates on the density.

\begin{Lemma}\label{lll1}
For any $T\in [0,T_{\textrm{max}}]$, the density $\varrho$ satisfies
\begin{equation*}
 \|\nabla\varrho\|_{L^\infty(0,T; L^{q}(\O))}\le CR^{\f{3}{2}}
\end{equation*}
where the positive constant $C$ is independent of $T$.
\end{Lemma}
\begin{proof}
We only need to prove
$$\|\sigma\|_{L^\infty(0,T; L^q(\O))}\le CR^{\f{3}{2}}.$$
Indeed, multiplying \eqref{xx2} by $|\sigma|^{q-2}\sigma$ and integrating over $\O$ yields
\begin{equation}\label{q5}
\f{1}{q}\f{d}{dt}\|\sigma(t)\|_{L^q(\O)}^q=-\int_\O\nabla\Dv\u\cdot\sigma |\sigma|^{q-2}-\int_\O\nabla(\u\cdot\sigma)\cdot\sigma|\sigma|^{q-2}dx.
\end{equation}
Note that
\begin{equation*}
 \begin{split}
  \left|\int_\O\nabla\Dv\u\cdot\sigma |\sigma|^{q-2}\right|\le \|\D\u\|_{L^q(\O)}\|\sigma\|_{L^q(\O)}^{q-1};
 \end{split}
\end{equation*}
and
\begin{equation*}
 \begin{split}
  \int_\O\nabla(\u\cdot\sigma)\cdot\sigma|\sigma|^{q-2}dx=\int_\O\partial_{x_j}\u_k\sigma_k\sigma_j|\sigma|^{q-2}dx+I
 \end{split}
\end{equation*}
where
\begin{equation*}
 \begin{split}
I&=:\int_\O\partial_{x_k}\partial_{x_j}\partial_{x_k}(\ln\varrho)\partial_{x_j}\ln\varrho |\sigma|^{q-2}dx\\
&=\f{1}{2}\int_\O\u_k\partial_{x_k}\left(\sum_{j=1}^3\sigma_j^2\right)|\sigma|^{q-2}dx=\int_\O\u_k(\partial_{x_k}|\sigma|)|\sigma|^{q-1}dx\\
&=\f{1}{q}\int_\O\u_k\partial_{x_k}(|\sigma|^{q})dx=-\f{1}{q}\int_\O|\sigma|^q\Dv\u dx.
\end{split}
\end{equation*}
Hence,
$$\left|\int_\O\nabla(\u\cdot\sigma)\cdot\sigma|\sigma|^{q-2}dx\right|\le C\|\nabla\u\|_{L^\infty}\|\sigma\|_{L^q}^q\le C\|\u\|_{W^{2,q}(\O)}\|\sigma\|_{L^q(\O)}^q.$$

Substituting those estimates back into \eqref{q5}, one obtains
\begin{equation*}
\f{1}{q}\f{d}{dt}\|\sigma(t)\|_{L^q(\O)}^q\le C\|\sigma\|_{L^q}^{q-1}\|\u\|_{W^{2,q}(\O)}\left(1+\|\sigma\|_{L^q}\right).
\end{equation*}
Multiplying the above inequality by $\|\sigma\|_{L^q}^{p-q}$ and integrating over $[0,T]$, we have, using Corrollary \ref{c1} and Lemma \ref{l51},
\begin{equation}\label{q6}
 \begin{split}
  \f{1}{q}\|\sigma(T)\|_{L^q}^p&\le\f{1}{q}\|\sigma_0\|_{L^q}^p+C\|\sigma\|_{L^p(0,T; L^q(\O))}^{p-1}\|\u\|_{L^p(0,T; W^{2,q}(\O))}\left(1+\|\sigma\|_{L^\infty(0, T; L^q(\O))}\right)\\
&\le\f{1}{q}\|\sigma_0\|_{L^q}^p+CR^{2p}\left(1+\|\sigma\|_{L^\infty(0, T; L^q(\O))}\right)\\
&\le CR^{2p}\left(1+\|\sigma\|_{L^\infty(0, T; L^q(\O))}\right).
 \end{split}
\end{equation}
Thanks to Lemma \ref{sigma}, the continuity of $t\mapsto \|\sigma(t)\|_{L^q(\O)}$ implies that $$\|\sigma(t)\|_{L^q(\O)}\le R^{\f{3}{2}}$$ in some maximal interval $(0, T(R))\subset(0,T_{\textrm{max}})$. If $T(R)<T_{\textrm{max}}$, then we have
$$\|\sigma(T(R))\|_{L^q(\O)}= R^{\f{3}{2}}.$$ On the other hand,
\eqref{q6} implies
\begin{equation*}
\begin{split}
 R^{\f{3}{2}}&=\|\sigma(T(R))\|_{L^q(\O)}\le C^{\f{1}{p}}R^2\left(1+\|\sigma\|_{L^\infty(0, T; L^q(\O))}\right)\\
&\le C^{\f{1}{p}}R^2(1+R^{\f{3}{2}}).
\end{split}
\end{equation*}
This is impossible, if we assume that $R$ is so small that
$$C^{\f{1}{p}}R^{\f{1}{2}}(1+R^{\f{3}{2}})<1.$$
Thus, this implies that
$$\|\sigma\|_{L^\infty(0,T_{\textrm{max}}; L^q(\O))}\le R^{\f{3}{2}}.$$
The proof is complete.
\end{proof}

Similarly, we have
$$\|\nabla E\|_{L^\infty(0,T_{\textrm{max}}; L^{q}(\O))}\le C R^{\f32},$$
and hence
\begin{equation}\label{h1}
\max\{\|\nabla\r\|_{L^\infty(0,T_{\textrm{max}};
L^{q}(\O))},\|\nabla E\|_{L^\infty(0,T_{\textrm{max}};
L^{q}(\O))}\}\le C R^{\f32}.
\end{equation}

On the other hand, Proposition \ref{p2}, combined with the initial conditions \eqref{a2} and \eqref{a3}, imply that
\begin{equation}\label{h2}
\int_\O\r dx=0,\quad\textrm{and}\quad \int_\O(1+\r)E_{ij}=0,\quad\textrm{for all}\quad i,j=1,2,3.
\end{equation}
Poincare's inequality, combined with \eqref{h1} and \eqref{h2}, yields
\begin{equation*}
\max\{\|\r\|_{L^\infty(0,T_{\textrm{max}};
L^{q}(\O))},\|E\|_{L^\infty(0,T_{\textrm{max}};
L^{q}(\O))}\}\le C R^{\f32},
\end{equation*}
and thus
\begin{equation}\label{h}
\max\{\|\r\|_{L^\infty(0,T_{\textrm{max}};
W^{1,q}(\O))},\|E\|_{L^\infty(0,T_{\textrm{max}};
W^{1,q}(\O))}\}\le C R^{\f32}.
\end{equation}

\medskip

Now we are in a position to finish the proof of Theorem \ref{T1}.

\begin{proof}[Proof of Theorem \ref{T1}]
 Suppose that $T_{\textrm{max}}<\infty$. Let $T_n\nearrow T_{\textrm{max}}$ be arbitrary. Then necessarily
$$\|\u\|_{\mathcal{W}(0,T_n)}\nearrow R\quad\textrm{as}\quad n\rightarrow \infty.$$
Indeed, if this was not the case, then
$$\sup_{n}\|\u\|_{\mathcal{W}(0,T_n)}<R$$
and the function $\u$, defined on $(0,T_{\textrm{max}})$, would be a solution to the equation \eqref{e1e} on $(0, T_{\textrm{max}})$. Taking $\u(T_{\textrm{max}})$ as a new
initial condition, from Lemma \ref{T3}, we get $\u(T_{\textrm{max}})\in V_0^{p,q}$. And then, according to Theorem \ref{T20}, there exists a $T_0$ such that there exists a unique solution to \eqref{e1e} on $(T_{\textrm{max}}, T_{\textrm{max}}+T_0)$. This means we can extend the solution we obtained outside the interval $(0, T_{\textrm{max}})$ to $(0, T_{\textrm{max}}+T_0)$
with $\|\u\|_{\mathcal{W}(0,T_{\textrm{max}}+T_0)}<R$ and
$$\max\{\|\r\|_{L^\infty(0,T_{\textrm{max}}+T_0; W^{1,q}(\O))},\| E\|_{L^\infty(0,T_{\textrm{max}}+T_0; W^{1,q}(\O))}\}< R$$ for some $T_0>0$. This is a contradiction with the maximality of $T_{\textrm{max}}$. Thus, it is impossible to have
$$T_{\textrm{max}}<\infty.$$ Hence, the solution is defined for all positive time.

Finally, from the
estimate \eqref{h}, one has
$$\|\r\|_{L^\infty(0,\infty; W^{1,q}(\O))}<R$$
and
$$\|E\|_{L^\infty(0,\infty; W^{1,q}(\O))}<R,.$$

The proof is complete.
\end{proof}

\bigskip\bigskip

\section*{Acknowledgments}
D. Wang's research was supported in part by the National Science Foundation under grant DMS-0906160,
and by the Office of Naval Research under grant N00014-07-1-0668.

\end{document}